\documentclass[a4paper,english,fontsize=11pt,parskip=half,abstract=true]{scrartcl}
\usepackage{babel}
\usepackage[utf8]{inputenc}
\usepackage[T1]{fontenc}
\usepackage[a4paper,left=20mm,right=20mm,top=30mm,bottom=30mm]{geometry}
\usepackage{amsmath}
\usepackage{amsthm}
\usepackage{amssymb}
\usepackage{enumerate}
\usepackage{thmtools}
\usepackage{booktabs}
\usepackage[bookmarks=true,
            pdftitle={Restrictions of characters in p-solvable groups},
            pdfauthor={Benjamin Sambale},
            pdfkeywords={},
            pdfstartview={FitH}]{hyperref}

\newtheorem{Thm}{Theorem} 
\newtheorem{Lem}[Thm]{Lemma}
\newtheorem{Prop}[Thm]{Proposition}
\newtheorem{Cor}[Thm]{Corollary}
\newtheorem{Con}[Thm]{Conjecture}
\theoremstyle{definition}

\numberwithin{equation}{section}

\allowdisplaybreaks[1]
\newenvironment{sm}{\bigl(\begin{smallmatrix}}{\end{smallmatrix}\bigr)}

\renewcommand{\phi}{\varphi}

\newcommand{\Z}{\mathrm{Z}}
\newcommand{\ZZ}{\mathbb{Z}}
\newcommand{\CC}{\mathbb{C}}

\newcommand{\FF}{\mathbb{F}}

\newcommand{\pcore}{\mathrm{O}}
\newcommand{\GL}{\operatorname{GL}}
\newcommand{\SL}{\operatorname{SL}}
\newcommand{\PSU}{\operatorname{PSU}}
\newcommand{\PSL}{\operatorname{PSL}}
\newcommand{\PGL}{\operatorname{PGL}}

\newcommand{\Sz}{\operatorname{Sz}}
\newcommand{\Irr}{\operatorname{Irr}}

\newcommand{\Syl}{\operatorname{Syl}}

\newcommand{\Ker}{\operatorname{Ker}}

\newcommand{\diag}{\operatorname{diag}}

\mathchardef\ordinarycolon\mathcode`\:  
 \mathcode`\:=\string"8000
 \begingroup \catcode`\:=\active
   \gdef:{\mathrel{\mathop\ordinarycolon}}
 \endgroup

\title{Restrictions of characters in $p$-solvable groups}
\author{Damiano Rossi\footnote{Arbeitsgruppe Algebra und Zahlentheorie, Bergische Universität Wuppertal, Gaußstraße~20, 42119 Wuppertal, Germany, \href{mailto:rossi@uni-wuppertal.de}{rossi@uni-wuppertal.de}} \ and Benjamin Sambale\footnote{Institut für Algebra, Zahlentheorie und Diskrete Mathematik, Leibniz Universität Hannover, Welfengarten~1, 30167 Hannover, Germany,
\href{mailto:sambale@math.uni-hannover.de}{sambale@math.uni-hannover.de}}}
\date{\today}

\begin{document}
\frenchspacing
\maketitle

\begin{abstract}\noindent
Let $G$ be a $p$-solvable group, $P\le G$ a $p$-subgroup and $\chi\in\Irr(G)$ such that $\chi(1)_p\ge|G:P|_p$. We prove that the restriction $\chi_P$ is a sum of characters induced from subgroups $Q\le P$ such that $\chi(1)_p=|G:Q|_p$. This generalizes previous results by Giannelli--Navarro and Giannelli--Sambale on the number of linear constituents of $\chi_P$. 
Although this statement does not hold for arbitrary groups, we conjecture a weaker version which can be seen as an extension of Brauer--Nesbitt's theorem on characters of $p$-defect zero. It also extends a conjecture of Wilde.
\end{abstract}


\section{Introduction}

Let $p$ be a prime and let $G$ be a finite group whose order is divisible by $p^n$, but not by $p^{n+1}$. Recall that an irreducible character $\chi\in\Irr(G)$ has $p$-\emph{defect} $0$, if the degree $\chi(1)$ is divisible by $p^n$. A well-known theorem by Brauer and Nesbitt~\cite[Theorem~1]{BrauerNesbitt} asserts that $\chi$ vanishes on all elements $g\in G$ of order divisible by $p$. Equivalently, the restriction of $\chi$ to a Sylow $p$-subgroup $P$ of $G$ is a multiple of the regular character of $P$. In particular, all irreducible characters $\theta\in\Irr(P)$ appear as constituents of $\chi_P$ with multiplicity at least $\theta(1)$ in this case. In \cite{GSn,GAn,GN}, Giannelli and Navarro investigated a more general situation where $\chi(1)$ is divisible by $p$ and $\chi_P$ has at least one linear constituent $\lambda\in\Irr(P)$. They conjectured (and proved in many cases) that $\chi_P$ has at least $p$ distinct linear constituents. For $p$-solvable groups $G$, they actually showed the stronger statement that $(\lambda_Q)^P$ is a summand of $\chi_P$ for some subgroup $Q\le P$ with index $|P:Q|=p$. 

In a subsequent paper~\cite{GS}, Giannelli and the second author studied the following blockwise version of the conjecture. Let $B$ be a $p$-block of $G$ with defect group $D$. Let $\chi\in\Irr(B)$ be of positive height and assume that $\chi_D$ has a linear constituent. Then $\chi_D$ has at least $p$ distinct linear constituents. Although we proved this conjecture in some cases, the $p$-solvable group case was left open at that time. 

Meanwhile we realized that the restriction to Sylow subgroups or to defect groups was unnecessary and perhaps misleading in the latter case. The aim of the present paper is to prove the following more general theorem and its corollary for $p$-solvable groups.

\begin{Thm}\label{main}
Let $P$ be a Sylow $p$-subgroup of a finite $p$-solvable group $G$. If $\chi\in\Irr(G)$, then $\chi_P$ is a sum of characters induced from subgroups $Q\le P$ such that $\chi(1)_p=|P:Q|$.
\end{Thm}

The statement of \autoref{main} was of course inspired by \cite[Theorems~B and C]{GN}.
By using the Mackey formula, we can extend the above result to arbitrary $p$-subgroups.

\begin{Cor}\label{cormain}
Let $P$ be a $p$-subgroup of a finite $p$-solvable group $G$. Let $\chi\in\Irr(G)$ such that $\chi(1)_p\ge|G:P|_p$. Then $\chi_P$ is a sum of characters induced from subgroups $Q\le P$ such that $\chi(1)_p=|G:Q|_p$.
\end{Cor}

In fact, under certain circumstances, $p$-solvable groups can be replaced by $\pi$-separable groups where $\pi$ is a set of primes. The details are outlined after the proof of \autoref{main} below. If $G=P$, then \autoref{main} is the well-known fact that characters of $p$-groups are monomial.

The following consequence of \autoref{main} was originally proven by Navarro (private communication).

\begin{Cor}\label{corsylow}
Let $P$ be a Sylow $p$-subgroup of a $p$-solvable group $G$. Then for every $\chi\in\Irr(G)$ there exist $Q\le P$ and a linear character $\lambda\in\Irr(Q)$ such that $\chi(1)_p=|P:Q|$ and $\lambda^P$ is a summand of $\chi_P$.
\end{Cor}

Using Kessar--Malle's theorem~\cite{KessarMalle} on the height zero conjecture, Giannelli--Navarro~\cite[Theorem~C]{GN} have shown that \autoref{corsylow} remains true for arbitrary groups $G$ whenever $P$ is abelian (here $Q$ can be chosen as a defect group of the block containing $\chi$).
The next corollary settles the open question in \cite{GS} mentioned above.

\begin{Cor}\label{corblock}
Let $B$ be a $p$-block of a $p$-solvable group $G$ with defect group $D$. Let $\chi\in\Irr(B)$ be of positive height such that $\chi_D$ has a linear constituent. Then $\chi_D$ has at least $p$ distinct linear constituents.
\end{Cor}

Initially, we believed that the statement of \autoref{main} could hold for arbitrary groups, until we found the counterexample $\PSU(5,2)$ for $p=2$. Nevertheless, we did not find any counterexamples to the following weaker version (which is algorithmically much easier to check).

\begin{Con}\label{con}
Let $P$ be a Sylow $p$-subgroup of a finite group $G$ and let $\chi\in\Irr(G)$. Then $\chi_P$ is an integral linear combination of characters induced from subgroups $Q\le P$ such that $\chi(1)_p=|P:Q|$.
\end{Con}

Just like \autoref{cormain}, the conjecture implies an analogue statement for arbitrary $p$-subgroups.

\autoref{con} relates to a strong form of Brauer's induction theorem due to Willems~\cite{WillemsBInd}. If $\chi$ is realized by a module with vertex $Q\le P$, then Willems' theorem implies that $\chi_P$ is an integral linear combination of characters induced from subgroups of the form $P\cap Q^g$ where $g\in G$. In particular, if $\chi$ has height $0$, then $|P:P\cap Q^g|\ge|P:Q|\ge\chi(1)_p$. Therefore, \autoref{con} holds whenever $P$ is abelian by Kessar--Malle~\cite{KessarMalle}. In general we cannot expect to find a vertex $Q$ such that $\chi(1)_p\le|P:Q|$ (see \cite{Cline}). 

If $P=\langle g\rangle\le G$ is an arbitrary cyclic $p$-subgroup, then there is only one choice for $Q$. In this case \autoref{con} implies that $\chi(g)=0$ as long as $\chi(1)_p>|G:P|_p$. This is a special case of Wilde's Conjecture~\cite[Conjecture~1.1]{Wilde} (for $p$-elements). Apart from $p$-solvable groups, Wilde proved his conjecture also for symmetric groups. 
If true, \autoref{con} would also generalize the Brauer--Nesbitt theorem on characters of $p$-defect $0$.

In the next section we prove \autoref{main} and its corollaries. Then, we study minimal counterexamples to \autoref{con}. 
Among other results, we prove \autoref{con} whenever $G$ is one of the simple groups $\PSL(2,q)$ or $\Sz(q)$ as well as if $P$ is a dihedral or quaternion $2$-group. In these cases the integral linear combination has positive coefficients so that the statement of \autoref{main} remains true.

\section{Proofs}

Our notation is standard and follows Huppert's book~\cite{HuppertChar} and Navarro's books~\cite{Navarro,Navarro2}. 
We will make use of $\pi$-special characters where $\pi$ is a set of primes (in fact, $\pi=\{p\}$ or $\pi=p'$). 
Although the definition of a $\pi$-special character $\chi$ is somewhat technical, the reader only needs to know that every prime divisor of $\chi(1)$ lies in $\pi$. 

\begin{Lem}\label{lemmain}
Let $P$ be a Sylow $p$-subgroup of a finite group $G$. If $\chi\in\Irr(G)$ is imprimitive and \autoref{main} holds for every proper subgroup of $G$, then \autoref{main} holds for $\chi$.
\end{Lem}

\begin{proof}
Since $\chi$ is imprimitive, there exist $H<G$ and $\psi\in\Irr(H)$ such that $\chi=\psi^G$.
We may choose $P$ such that $S:=H\cap P$ is a Sylow $p$-subgroup of $H$. Since \autoref{main} holds for $H$, we may write $\psi_S=\sum_{i=1}^n\lambda_i^S$ where $\lambda_i\in\Irr(Q_i)$ and $Q_i\le S$ such that $\psi(1)_p=|S:Q_i|$ for $i=1,\ldots,n$. By Sylow's theorem, there exist representatives $g_1,\ldots,g_k\in G$ for the double cosets $H\backslash G/P$ such that $S_i:=H^{g_i}\cap P\le S^{g_i}$ for $i=1,\ldots,k$. The Mackey formula (see \cite[Theorem~17.4]{HuppertChar}) yields
\[\chi_P=(\psi^G)_P=\sum_{i=1}^k\bigl((\psi^{g_i})_{S_i}\bigr)^P=\sum_{i=1}^k\bigl(((\psi_S)^{g_i})_{S_i}\bigr)^P=\sum_{i=1}^k\sum_{j=1}^n\bigl(((\lambda_j^{g_i})^{S^{g_i}})_{S_i}\bigr)^P.\]
Here, $\bigl((\lambda_j^{g_i})^{S^{g_i}}\bigr)_{S_i}$ is a sum of characters induced from $Q_j^{g_ix}\cap S_i$ for some $x\in S^{g_i}$.
Since 
\[\chi(1)_p=|G:H|_p\psi(1)_p=|G:Q_j|_p\le |P:Q_j^{g_ix}\cap S_i|,\] 
we can find $Q_j^{g_ix}\cap S_i\le T_{ij}^{(x)}\le P$ such that $\chi(1)_p=|P:T_{ij}^{(x)}|$. Hence, the claim holds for $\chi$.
\end{proof}

By using \autoref{lemmain} and Isaacs' theory of $\pi$-special characters, we can now prove \autoref{main}.

\begin{proof}[Proof of \autoref{main}.]
We proceed by induction on the order of $G$. By \autoref{lemmain}, we can assume $\chi$ to be primitive. By a theorem of Isaacs (see \cite[Theorem~40.8]{HuppertChar}), $\psi=\alpha\beta$ where $\alpha$ is $p$-special and $\beta$ is $p'$-special. Thus, $\chi(1)_p=\alpha(1)$. By a theorem of Gajendragadkar (see \cite[Theorem~40.11]{HuppertChar}), we deduce that $\alpha_P\in\Irr(P)$. Since $P$ is an M-group, there exists $T\le P$ and a linear character $\mu\in\Irr(T)$ such that $\mu^P=\alpha_P$. It follows that
$$\chi_P=\alpha_P\beta_P=\mu^P\beta_P=(\mu\beta_T)^P,$$
where the last equality follows by \cite[Theorem~17.3]{HuppertChar}. Since $\chi(1)_p=\alpha(1)=|P:T|$ the result follows.
\end{proof}

Given the proof above, it is natural to ask whether we can replace $p$ by an arbitrary set of primes $\pi$ and $p$-solvable groups by $\pi$-separable groups. The only additional hypotheses we require is that every $\pi$-subgroup of $G$ is an M-group and we can always find the groups $T_{ij}^{(x)}$ in the proof of \autoref{lemmain}. Both requirements are true whenever $G$ has nilpotent $\pi$-Hall subgroups. 

Next, we obtain \autoref{cormain} from \autoref{main} via an application of the Mackey formula.

\begin{proof}[Proof of \autoref{cormain}.]
Let $P$ be a $p$-subgroup of $G$ and $\chi\in\Irr(G)$ with $\chi(1)_p\ge|G:P|_p$. Consider $S\in\Syl_p(G)$ such that $P\le S$. By \autoref{main} there exist $Q_i\le S$ and characters $\lambda_i$ of $Q_i$ such that $\chi(1)_p=|S:Q_i|$ and $\chi_S=\sum_i \lambda_i^S$. By the Mackey formula $\chi_P$ is a sum of the characters $(\lambda_i^S)_P=\sum((\lambda_i^x)_{Q_i^x\cap P})^P$, for some $x\in S$. For every such $x$, we find $Q_i^x\cap P\le R_{i,x}\le P$ such that $\chi(1)_p=|G:R_{i,x}|_p$. Set $\mu_{i,x}:=((\lambda_i^x)_{Q_i^x\cap P})^{R_{i,x}}$. Now $\chi_P$ is a sum of the characters $\mu_{i,x}^P$.
\end{proof}

\begin{proof}[Proof of \autoref{corsylow}.]
By \autoref{main}, there exist subgroups $Q_i\leq P$ and characters $\lambda_i\in\Irr(Q_i)$ such that $\chi(1)_p=|P:Q_i|$ and $\chi_P=\sum_i\lambda_i^P$. Then $\chi(1)_p=\chi(1)_p\sum_i\lambda_i(1)$ and hence there exists $j$ such that $\lambda_j$ is linear. Since $\lambda_j^P$ is a summand of $\chi_P$ the result follows.
\end{proof}

\begin{proof}[Proof of \autoref{corblock}.]
We apply \autoref{cormain} with $P=D$. Let $\theta\in\Irr(D)$ be a linear constituent of $\chi_D$. Then there exist $Q\le D$ and $\lambda\in\Irr(Q)$ such that $\chi(1)_p=|G:Q|_p$, $\theta$ occurs in $\lambda^D$ and $\lambda^D$ is a summand of $\chi_D$. By Frobenius reciprocity, this yields $\lambda=\theta_Q$ and $(\theta_Q)^D$ is a summand of $\chi_D$. 
Since $\chi$ has positive height, $\chi(1)_p>|G:D|_p$ and therefore $Q<D$. Choose $Q\le R\le D$ such that $|D:R|=p$. 
By Gallagher's theorem (see \cite[Theorem~19.5]{HuppertChar}), 
$(\theta_R)^D$ is a sum of $p$ distinct linear characters and they all appear in $\chi_D$ since $\theta_R$ is a constituent of $(\theta_Q)^R$. 
\end{proof}

\section{Evidence for \autoref{con}}\label{secevidences}

In this section we collect some evidence for \autoref{con}. We start with an analysis of minimal counterexamples. 

\begin{Thm}\label{mincounter}
Let $G$ be a minimal counterexample to \autoref{con} subject to $(|G:\Z(G)|,|G|)$. Then the following holds:
\begin{enumerate}[(a)]
\item\label{pp} $G=\pcore^{p'}(G)$.
\item\label{cp} $G$ is not a direct product of proper subgroups.
\item\label{prim} $\chi$ is primitive.
\item\label{faith} $\chi$ is faithful.
\item\label{cencyc} Every abelian normal subgroup of $G$ is cyclic and central.
\item\label{chartrip} $\pcore_{p'}(G)\le\Z(G)\cap G'$.
\end{enumerate}
\end{Thm}
\begin{proof}
First note that \autoref{con} holds for all proper subgroups $H<G$ and all proper quotients $G/N$ since
\begin{align*}
|H:\Z(H)|&\le|H:H\cap\Z(G)|=|H\Z(G):\Z(G)|\le|G:\Z(G)|,\\
|G/N:\Z(G/N)|&\le|G:\Z(G)N|\le|G:\Z(G)|.
\end{align*}
\begin{enumerate}[(a)]
\item Suppose that $N:=\pcore^{p'}(G)<G$. Then $P\le N$. 
By Clifford theory, $\chi_N$ is a sum of characters $\psi\in\Irr(N)$ such that $\chi(1)_p=\psi(1)_p$. 
Since \autoref{con} holds for $\psi$, it must also hold for $\chi$ contradicting the choice of $G$. Consequently, $N=G$.

\item Suppose that $G=G_1\times G_2$ for proper subgroups $G_1$ and $G_2$. Let $P=P_1\times P_2$ and $\chi=\chi_1\times\chi_2$ with $P_i\in\Syl_p(G_i)$ and $\chi_i\in\Irr(G_i)$ for $i=1,2$. 
By hypothesis, we can write $(\chi_i)_{P_i}=\sum_{k=1}^{n_i}a_{ik}\lambda_{ik}^{P_i}$ with $a_{ik}\in\ZZ$ and $\lambda_{ik}\in\Irr(Q_{ik})$ where $Q_{ik}\le P_i$ and $\chi_i(1)_p=|P_i:Q_{ik}|$. Then
\[\chi_P=(\chi_1)_{P_1}\times(\chi_2)_{P_2}=\sum_{k=1}^{n_1}\sum_{l=1}^{n_2}a_{1k}\lambda_{1k}^{P_1}\times a_{2l}\lambda_{2l}^{P_2}=\sum_{k,l}a_{1k}a_{2l}(\lambda_{1k}\times\lambda_{2l})^P\]
and $\chi(1)_p=\chi_1(1)_p\chi_2(1)_p=|P_1:Q_{1k}||P_2:Q_{2l}|=|P:Q_{1k}\times Q_{2l}|_p$ for all $k,l$. This means that \autoref{con} holds for $G$, a contradiction. Consequently, $G$ is not a direct product of proper subgroups.

\item This can be shown as in \autoref{lemmain}.

\item Suppose that $N:=\Ker(\chi)\ne 1$. We may regard $\chi$ as a character of $\overline{G}:=G/N$. 
Since $\overline{P}:=PN/N$ is a Sylow $p$-subgroup of $\overline{G}$, there exist $\overline{Q_i}\le\overline{P}$ and $\lambda_i\in\Irr(\overline{Q_i})$ such that $\chi_{\overline{P}}=\sum_{i=1}^na_i\lambda_i^{\overline{P}}$ for some $a_i\in\ZZ$. Choose $Q_i\le P$ such that $\overline{Q_i}=Q_iN/N$ and $Q_i=Q_iN\cap P$ for $i=1,\ldots,n$. Note that $P\cap N=Q_i\cap N$.
Using the canonical isomorphism $Q_i/Q_i\cap N\to\overline{Q_i}$, we can identify $\lambda_i$ with the corresponding character of $Q_i$. This allow us to write $\chi_P=\sum_ia_i\lambda_i^P$. Since $|Q_iN:Q_i|=|Q_iN:Q_iN\cap P|=|PN:P|$, $Q_i$ is a Sylow $p$-subgroup of $Q_iN$. It follows that 
\[\chi(1)_p=|\overline{G}:\overline{Q_i}|_p=|G:Q_iN|_p=|P:Q_i|.\]
Hence, the claim would hold for $G$. Consequently, $\chi$ must be faithful.

\item Let $N\unlhd G$ be abelian. By \eqref{prim}, $\chi_N=\chi(1)\theta$ for some linear $\theta\in\Irr(N)$. By \eqref{faith}, $\chi$ is faithful and so must be $\theta$. Hence, $N\cong\theta(N)\le\CC^\times$ is cyclic. 
As $\theta$ is $G$-invariant, we obtain $N\le\Z(G)$. 

\item By \eqref{pp} we have $N:=\pcore_{p'}(G)\le G'$. By \eqref{prim}, there is a unique $\theta\in\Irr(N)$ under $\chi$. 
Suppose that $N\nsubseteq\Z(G)$.
By \cite[Problem (6.3)]{Navarro2}, there exists a character triple isomorphism $(G,N,\theta)\cong(G^*,N^*,\theta^*)$ such that $N^*$ is a central $p'$-subgroup of $G^*$ and $\theta^*$ is linear. Since $\pcore_{p'}(G^*/N^*)\cong\pcore_{p'}(G/N)=1$, we have $N^*=\pcore_{p'}(G^*)$. 
It is easy to see that $\Z(G^*/N^*)=\Z(G^*)/N^*$. In particular, 
\[|G^*:\Z(G^*)|=|G/N:\Z(G/N)|\le|G:\Z(G)N|<|G:\Z(G)|,\]
and so the claim holds for $\chi^*$. Also notice that $\chi(1)=\chi^*(1)\theta(1)$ and $\chi(1)_p=\chi^*(1)_p$ (see \cite[p. 87]{Navarro2}).
Let $P^*$ be a Sylow $p$-subgroup of $G^*$ such that $(PN/N)^*=P^*N/N$. Then $(\chi^*)_{P^*}=\sum_ia_i\mu_i^{P^*}$ with $\mu_i\in\Irr(Q_i^*)$, $Q_i^*\le P^*$ and $\chi^*(1)_p=|G^*:Q_i^*|_p$. 
Choose $Q_i\le P$ and $\lambda_i\in\Irr(Q_iN|\theta)$ such that $(Q_iN/N)^*=Q_i^*N^*/N^*$ and $\lambda_i^*=\mu_i\times\theta^*$. Now by \cite[p.~87]{Navarro2}, 
\begin{align*}
(\chi_{PN})^*&=(\chi^*)_{P^*N^*}=(\chi^*)_{P^*}\times\theta^*=\sum_i a_i\mu_i^{P^*}\times\theta^*\\
&=\sum_i a_i(\mu_i\times\theta^*)^{P^*N^*}=\sum_i a_i(\lambda_i^*)^{P^*N^*}=\sum_i(a_i\lambda_i^{PN})^*.
\end{align*}
Since $\psi\mapsto\psi^*$ is a bijection between $\mathbb{Z}\Irr(PN|\theta)$ and $\mathbb{Z}\Irr(P^*N^*|\theta^*)$, we obtain 
\[\chi_P=(\chi_{PN})_P=\sum_i(a_i\lambda_i^{PN})_P=\sum_ia_i(\lambda_Q)^P.\]
Moreover, $\chi(1)_p=\chi^*(1)_p=|G^*:Q^*|_p=|G:Q|_p$.
This contradiction finally shows that $N$ is central. 
\qedhere
\end{enumerate}
\end{proof}

Our next goal, as mentioned in the introduction, is to prove the statement of \autoref{main} for the simple groups $\PSL(2,q)$ and $\Sz(q)$. We refer to this statement as the \emph{strong form} of \autoref{con} (we remind the reader that there are counterexamples to this stronger claim). 
The advantage of working with the strong form of the conjecture is that the claim carries over to quotients.

\begin{Prop}\label{normal}
Let $N\unlhd G$. If the strong form of \autoref{con} holds for $G$, then the same is true for $G/N$.
\end{Prop}
\begin{proof}
We use the bar convention $\overline{H}:=HN/N$ for $H\le G$. Let $P$ be a Sylow $p$-subgroup of $G$, so that $\overline{P}$ is a Sylow $p$-subgroup of $\overline{G}$. We identify the characters $\chi\in\Irr(\overline{G})$ with their inflation to $G$. By hypothesis, $\chi_P$ is a sum of characters $\lambda^P$ where $\lambda\in\Irr(Q)$ and $Q\le P$ such that $\chi(1)_p=|G:Q|_p$. 
Since
\[P\cap N\le\Ker(\chi_P)\le\Ker(\lambda^P)=\bigcap_{x\in P}\Ker(\lambda)^x\le\Ker(\lambda),\]
we can consider $\lambda$ as a character of $\overline{Q}\cong Q/Q\cap N$. Moreover $P\cap N=Q\cap N$. For $x\in P$ we compute
\begin{align*}
\lambda^{\overline{P}}(xN)&=\frac{1}{|\overline{Q}|}\sum_{\substack{yN\in\overline{P}\\x^yN\in\overline{Q}}}\lambda(x^yN)=\frac{|Q\cap N|}{|Q|}\sum_{\substack{y(P\cap N)\in P/P\cap N\\x^y\in Q}}\lambda(x^y)\\
&=\frac{|Q\cap N|}{|Q||P\cap N|}\sum_{\substack{y\in P\\x^y\in Q}}\lambda(x^y)=\frac{1}{|Q|}\sum_{\substack{y\in P\\x^y\in Q}}\lambda(x^y)=\lambda^P(x).
\end{align*}
Hence, $\chi_{\overline{P}}$ is the sum of the induced characters $\lambda^{\overline{P}}$. Since $Q\cap N=P\cap N$ is a Sylow $p$-subgroup of $N$, we also obtain $\chi(1)_p=|G:Q|_p=|G:QN|_p=|\overline{G}:\overline{Q}|_p$.
\end{proof}

In the introduction we explained how \autoref{con} follows from Kessar--Malle~\cite{KessarMalle} when $P$ is abelian. If $P$ is cyclic, the strong form of \autoref{con} can be shown without using the classification of finite simple groups.

\begin{Prop}\label{cyclic}
If $P$ is cyclic, then the strong form of \autoref{con} holds for $G$. 
\end{Prop}
\begin{proof}
Let $B$ be the $p$-block of $G$ containing $\chi$. Let $D\le P$ be a defect group of $B$. 
Since the elements of $P\setminus D$ are not conjugate to elements of $D$, $\chi_P$ vanishes outside $D$ by \cite[Corollary~5.9]{Navarro}. By the definition of character induction, it follows that 
\[\chi_P=\frac{1}{|P:D|}(\chi_D)^P.\] 
Let $\lambda\in\Irr(D)$ be a constituent of $\chi_D$. Then $\lambda$ extends to a (linear) character $\hat\lambda\in\Irr(P)$.
By Frobenius reciprocity,
\[[\lambda,\chi_D]=[\hat\lambda,(\chi_D)^P]=|P:D|[\hat\lambda,\chi_P]\]
is divisible by $|P:D|$. Hence, $\frac{1}{|P:D|}\chi_D$ is a proper character of $D$. 
Since $\chi$ has height zero in $B$ (this was known to Brauer~\cite[6C]{BrauerBlSec2} and does not require \cite{KessarMalle}), we have $\chi(1)_p=|G:D|_p$. 
\end{proof}

We can now prove the claimed result for the groups $\PSL(2,q)$ and $\Sz(q)$.

\begin{Prop}\label{lineargroups}
The strong form of \autoref{con} holds for $G=\SL(2,q)$, $\PSL(2,q)$ and $\Sz(2^{2n+1})$, where $q$ is a prime power and $n\geq 1$.
\end{Prop}
\begin{proof}
By \autoref{normal} it is enough to consider $G=\SL(2,q)$ and $\Sz(2^{2n+1})$. We start by considering $G=\SL(2,q)$. If $2<p\nmid q$, then $G$ has cyclic Sylow $p$-subgroups and the claim follows from \autoref{cyclic}. Let $p\mid q$. The Steinberg character $\chi\in\Irr(G)$ has $p$-defect zero and fulfills the claim with $Q=1$ by Brauer--Nesbitt's theorem mentioned in the introduction. Every other character $\chi\in\Irr(G)$ has $p'$-degree and fulfills the claim for trivial reasons. Finally, let $p=2\nmid q$. The Sylow $2$-subgroup 
\[P=\langle x,y\mid x^{2^n}=y^4=1,\ x^y=x^{-1}\rangle\] 
is a quaternion group of order $2^{n+1}=(q^2-1)_2$ (recall that $|G|=q^3-q$). Let $X:=\langle x\rangle$ and $\lambda\in\Irr(X)$ be faithful. Then $\Irr(P)$ consists of four linear characters and the induced characters $(\lambda^k)^P$ of degree $2$ for $k=1,\ldots,2^{n-1}-1$. 
The character table of $G$ depends on $q\equiv\pm1\pmod{4}$ (equivalently $q\equiv\pm1\pmod{2^n}$).
It suffices to consider $\chi\in\Irr(G)$ with even degree. Since $y$ is conjugate to $x^{2^{n-2}}$ in $G$, the restrictions $\chi_P$ are determined by the following values:
\[
\begin{array}{c|ccc}
&1&z&x^j\\\midrule
\chi_k&q\pm1&(-1)^k(q\pm1)&\pm(\lambda^k)^P(x^j)\\
\psi_1&q\mp1&-(q\mp1)&0\\
\psi_2&\frac{q\mp1}{2}&-\frac{q\mp1}{2}&0
\end{array}
\]
where $z=x^{2^{n-1}}$ and $1\le j,k\le 2^{n-1}-1$. Let $Z:=\langle z\rangle=\Z(P)=\Z(G)$ and $W:=\langle x^{2^{n-2}}\rangle\cong C_4$. Then $\psi_1=\frac{q\mp1}{2^n}(\lambda_Z)^P$ and $\psi_2=\frac{q\mp1}{2^n}(\lambda_W)^P$ fulfill the claim since $\psi_1(1)_2=(q\mp1)_2=2^n=|P:Z|$ and $\psi_2(1)_2=2^{n-1}=|P:W|$. 
Now suppose that $k$ is odd, so that $(\chi_k)_Z=(q\pm1)\lambda_Z$. Then $\chi_k$ has no linear constituents, since those lie over $1_Z$. Hence, $\chi_k$ is a sum of characters induced from subgroups of index $2=(q\pm1)_2=\chi_k(1)_2$. 

The case $k\equiv 0\pmod{2}$ is more complicated. Here $\chi_k(y)=\pm(-1)^{k/2}2=:(-1)^s2$. Let $Y_1:=\langle x^2,y\rangle\cong Q_{2^n}$, $Y_2:=\langle x^2,xy\rangle\cong Q_{2^n}$ and $\mu_i\in\Irr(Y_i)$ such that $\mu_i(x^2)=1$ and $\mu_1(y)=\mu_2(xy)=-1$. We compute
\[\chi_k=\frac{q\mp1}{2^n}\bigl(1_Z\bigr)^P\pm(\lambda^k)^P+(\mu_1^s)^P+(\mu_2^s)^P-(1_{P'})^P,\]
where we recall that $P'=\langle x^2\rangle$.
Since $k$ is even, $\lambda^k$ lies over $1_Z$ and so does $1_{P'}$. On the other hand, $\lambda^k$ does not lie over $1_{P'}$ since $k<2^{n-1}$. It follows that $(1_Z)^X\pm\lambda^k-1_{P'}^X$ is a proper character of $X$ and $\chi_k$ is a sum of characters induced from $X$, $Y_1$ and $Y_2$. All have index $2=\chi_k(1)_2$ in $P$. 

Now let $G=\Sz(q)$ where $q=2^{2n+1}$ and $n\geq 1$. The Sylow $p$-subgroups for $p>2$ are cyclic by \cite[Theorem~9]{Suzukidouble}. Thus, again we restrict to $p=2$. The Sylow $2$-subgroup $P$ is a so-called Suzuki $2$-group of order $q^2$ such that $Z:=\Z(P)=P'=\Phi(P)$ is elementary abelian of order $q$ and contains all involutions of $P$ (see \cite[Theorem~7]{Suzukidouble}).
The character table of $G$ is given in \cite[Theorem~13]{Suzukidouble}. It can be seen that there are only two characters $\chi,\overline{\chi}\in\Irr(G)$ of even degree and not of $2$-defect $0$.
The values on $P$ are $\chi(1)=2^n(q-1)$, $\chi(z)=-2^n$ and $\chi(x)=2^n\sqrt{-1}$ where $z\in Z\setminus\{1\}$ and $x\in P\setminus Z$. 
This implies $\chi_Z=2^n(\rho_Z-1_Z)$ where $\rho_Z$ is the regular character of $Z$. Therefore, $\chi_P$ has no linear constituents, because those must lie over $1_{P'}=1_Z$. On the other hand, it has been shown in \cite{Sagirov} that all non-linear characters of $P$ have degree $2^n$. Hence, $\chi_P$ is a sum of characters induced from subgroups of index $2^n=\chi(1)_2$.
\end{proof}

As a final result we verify the strong form of \autoref{con} if $P$ is a dihedral or quaternion $2$-group. Notice that the proof of \autoref{mincounter} applies verbatim to the strong form. This remark will be used in the following.

\begin{Prop}\label{dihedral}
Let $P$ be a Sylow $p$-subgroup of $G$ and suppose that $P$ is a dihedral or quaternion $2$-group including the Klein four-group. Then the strong form of \autoref{con} holds for $G$. 
\end{Prop}

\begin{proof}
Let $G$ be a finite group with dihedral or quaternion Sylow $2$-subgroup $P$. By \autoref{main}, we may assume that $G$ is non-solvable. We argue by induction on $(|G:\Z(G)|,|G|)$. Since every subgroup and every quotient of $P$ is a cyclic, dihedral or quaternion group, we may apply the reduction methods from \autoref{mincounter} (the character triple isomorphism in \autoref{mincounter}\eqref{chartrip} preserves $P$). Specifically, we assume that $\pcore_{2'}(G)\le\Z(G)\cap G'$,  $G=\pcore^{2'}(G)$ and $\chi$ is primitive (but not necessarily faithful). Now the Gorenstein--Walter theorem shows that $G$ is a Schur cover of $A_7$, $\PSL(2,q)$ or of $\PGL(2,q)$ where $q>3$ is an odd prime power (see \cite[Theorems~6.8.7 and 6.8.9]{Suzuki2}).
The first case and the exceptional cover $3.\PSL(2,9)=3.A_7$ can be checked by computer while $\PSL(2,q)$ has been considered in \autoref{lineargroups}. It remains to prove the claim for $\PGL(2,q)$. Thanks to \autoref{normal} it suffices to consider $G=\GL(2,q)$. The reader can find the following well-known facts in \cite[§5.2]{FultonHarris}, for instance.

\textbf{Case~1:} $G=\GL(2,q)$ with $q\equiv -1\pmod{4}$.\\
Here $P=\langle x,y\mid x^{2^n}=y^2=1,\ x^y=x^{2^{n-1}-1}\rangle$ is a semidihedral group of order $2^{n+1}$ where $2^n=(q^2-1)_2$. Let $z:=x^{2^{n-1}}$ be the central involution in $G$. All non-central involutions are conjugate to $y$ and all elements of order $4$ are conjugate to $x^{2^{n-2}}$ in $G$. The characters of degree $q+1$ are induced from proper subgroups. 
There remains only one family of characters to consider. Let $X:=\langle x\rangle$ and $\lambda\in\Irr(X)$ be faithful. 
The restrictions $\chi_P$ assume the following values
\[
\begin{array}{c|cccc}
&1&z&x^j&y\\\midrule
\chi_k&q-1&(-1)^k(q-1)&-(\lambda^k)^P(x^j)&0
\end{array}
\]
where $k\not\equiv 0\pmod{2^{n-1}}$. 
If $k$ is odd, then $\chi_k$ does not have linear constituents and the claim follows since $\chi_k(1)_2=(q-1)_2=2$. If on the other hand $k\equiv 0\pmod{2}$, then
\[\chi_k=\frac{q+1}{2^{n-1}}(\lambda_W)^P-\lambda^P=\Bigl(\frac{q+1}{2^{n-1}}(\lambda_W)^X-\lambda\Bigr)^P\]
where $W:=\langle x^{2^{n-2}}\rangle\cong C_4$. The claim follows as before.

\textbf{Case~2:} $G=\GL(2,q)$ with $q\equiv 1\pmod{4}$.\\
In this case 
\[P=\langle x,y,z\mid x^{2^n}=y^{2^n}=z^2=[x,y]=1,\ x^z=y\rangle\cong C_{2^n}\wr C_2,\]
where $2^n=(q-1)_2$. This group can be realized conveniently by $x=\diag(\zeta,1)$, $y=\diag(1,\zeta)$ and $z=\begin{sm}
0&1\\1&0
\end{sm}$ where $\zeta\in\FF_q^\times$ has order $2^n$. 
Note that $xy\in\Z(P)\le\Z(G)$ and $z$ is conjugate to $x^{2^{n-1}}$ in $G$. Moreover, $xz$ has order $2^{n+1}$ and is conjugate to $(xz)^{2^n+1}=-xz$. Also, $P/Z\cong D_{2^{n+1}}$.
Let $\lambda\in\Irr(\langle xz\rangle)$ be faithful. The values of $\chi_P$ are
\[
\begin{array}{c|cccc}
&1&(xy)^i&x^iy^j&(xz)^k\\\midrule
\chi_l&q-1&(q-1)\lambda^l(xy)^i&0&-(\lambda^l)^P((xz)^k)
\end{array}
\]
where $i\not\equiv j\pmod{2^n}$, $k\equiv 1\pmod{2}$ and $l\not\equiv 0\pmod{2^{n-1}}$. 
If $l$ is odd, then $\chi_l$ vanishes on $(xz)^k$. Thus, $\chi_l=\frac{q-1}{2^n}(\lambda^l)^P$ and $\chi_l(1)_2=(q-1)_2=2^n=|P:\langle xz\rangle|$. If $l$ is even, then $\chi_l(xz)=-2\lambda^l(xz)=(\lambda^{l+2^n})^P(xz)$ and
\[\chi_l=\frac{q-1-2^n}{2^{n+1}}(\lambda^l)^P+\frac{q-1+2^n}{2^{n+1}}(\lambda^{l+2^n})^P.\qedhere\]
\end{proof}

By analyzing the Ree groups $G={^2G_2(q)}$ and making use of Walter's theorem, the strong form of \autoref{con} can be shown for all groups with abelian Sylow $2$-subgroups. Since \autoref{con} holds for abelian $P$, we omit the details.

Using GAP~\cite{GAP48}, \texttt{4ti2}~\cite{4ti2} and its GAP interface~\cite{GAP4ti2}, we checked the strong form of \autoref{con} for all groups of order at most $2000$. Moreover, none of the perfect groups of order at most $10^6$ are minimal counterexamples in the sense of \autoref{mincounter}.
Additionally, the stated form of \autoref{con} has been verified for all simple groups up to $\PSL(3,13)$ and all sporadic groups up to $Co_3$ (with respect to the group order). 

\section*{Acknowledgment}
We thank Gabriel Navarro for sharing his insights on a previous version of this paper, and for requesting more computer checking.
We appreciate Eugenio Giannelli's effort to prove corresponding results for symmetric groups. Moreover, Alexander Hulpke has kindly provided an updated database~\cite{Hulpke} of all perfect groups of order at most $10^6$. Thomas Breuer has introduced the authors to numerous tricks regarding character tables in GAP. The first author is supported by the research training group GRK2240: Algebro-geometric Methods in Algebra, Arithmetic and Topology of the German Research Foundation.
The second author is supported by the German Research Foundation (\mbox{SA 2864/1-2} and \mbox{SA 2864/3-1}).

\end{document}